\documentclass[10pt]{article}
\def\version{July 13, 2012}

\nonstopmode

\usepackage{bm}
\usepackage{upgreek}
\usepackage{mathrsfs}
\usepackage{amsthm}
\usepackage{amssymb}
\usepackage[dvips]{graphicx}
\DeclareGraphicsExtensions{.eps,.eps.gz}
\usepackage{amsmath}

\providecommand{\eqref}[1]{{\rm (\ref{#1})}}

\DeclareSymbolFont{EUR}{U}{eur}{m}{n}
\SetSymbolFont{EUR}{bold}{U}{eur}{b}{n}
\DeclareSymbolFontAlphabet{\eur}{EUR}

\DeclareSymbolFont{EUB}{U}{eur}{b}{n}
\SetSymbolFont{EUB}{bold}{U}{eur}{b}{n}
\DeclareSymbolFontAlphabet{\eub}{EUB}

\DeclareSymbolFont{AMSb}{U}{msb}{m}{n}
\DeclareSymbolFontAlphabet{\mathbb}{AMSb}

\newcommand{\eubJ}{\eub{J}}
\newcommand{\eubL}{\eub{L}}

\newcommand{\eurL}{\eur{L}}

\newcommand{\notyet}[1]{{}}

\newcommand{\range}{\mathop{\rm Range\,}}
\newcommand{\supp}{\mathop{\rm supp}}

\newcommand{\p}{\partial}
\newcommand{\at}[1]{\vert\sb{\sb{#1}}}
\newcommand{\At}[1]{\biggr\vert\sb{\sb{#1}}}

\def\R{\mathbb{R}}
\newcommand{\C}{\mathbb{C}}

\newcommand{\N}{\mathbb{N}}

\newcommand{\abs}[1]{\vert #1 \vert}

\newcommand{\norm}[1]{\Vert #1 \Vert}

\newcommand{\sothat}{{\rm ;}\ }

\DeclareMathSymbol{\varGamma}{\mathord}{letters}{"00}
\DeclareMathSymbol{\varDelta}{\mathord}{letters}{"01}
\DeclareMathSymbol{\varTheta}{\mathord}{letters}{"02}
\DeclareMathSymbol{\varLambda}{\mathord}{letters}{"03}
\DeclareMathSymbol{\varXi}{\mathord}{letters}{"04}
\DeclareMathSymbol{\varPi}{\mathord}{letters}{"05}
\DeclareMathSymbol{\varSigma}{\mathord}{letters}{"06}
\DeclareMathSymbol{\varUpsilon}{\mathord}{letters}{"07}
\DeclareMathSymbol{\varPhi}{\mathord}{letters}{"08}
\DeclareMathSymbol{\varPsi}{\mathord}{letters}{"09}
\DeclareMathSymbol{\varOmega}{\mathord}{letters}{"0A}

\theoremstyle{plain}
\newtheorem{lemma}{Lemma}[section]
\newtheorem{theorem}[lemma]{Theorem}

\newtheorem{proposition}[lemma]{Proposition}

\theoremstyle{definition}
\newtheorem{definition}[lemma]{Definition}

\theoremstyle{remark}
\newtheorem{remark}[lemma]{Remark}

\makeatletter\@addtoreset{equation}{section}
\makeatletter\@addtoreset{lemma}{section}
\makeatother

\def\spec{\sigma}

\renewcommand{\Re}{\mathop{\rm{R\hskip -1pt e}}\nolimits}
\renewcommand{\Im}{\mathop{\rm{I\hskip -1pt m}}\nolimits}

\begin{document}

\title{
Linear instability
of nonlinear Dirac equation in 1D
with higher order nonlinearity
}

\author{
{\sc Andrew Comech}
\\
{\it\small Texas A\&M University, College Station, TX 77843, U.S.A.}
\\
{\it\small Institute for Information Transmission Problems,
Moscow 101447, Russia}
}

\date{\version}

\maketitle

\begin{abstract}
We consider the nonlinear Dirac equation
in one dimension,
also known as
the Soler model in (1+1) dimensions,
or the massive Gross-Neveu model:
\[
i\p\sb t\psi=-i\alpha\p\sb x\psi
+m\beta\psi-f(\psi\sp\ast\beta\psi)\beta\psi,
\qquad
\psi(x,t)\in\C^2,
\quad
x\in\R;
\qquad
g\in C\sp\infty(\R),
\quad
m>0,
\]
where $\alpha$, $\beta$
are $2\times 2$ hermitian matrices
which satisfy
$\alpha^2=\beta^2=1$,
$\alpha\beta+\beta\alpha=0$.
We study the spectral stability of solitary wave
solutions $\phi\sb\omega(x)e^{-i\omega t}$.
More precisely, 
we study
the presence of point eigenvalues
in the spectra
of linearizations at solitary waves of arbitrarily small amplitude,
in the limit $\omega\to m$.
We prove that
if $f(s)=s^k+O(s^{k+1})$, $k\in\N$,
with $k\ge 3$,
then
one positive and one negative eigenvalue
are present in the spectrum of linearizations
at all solitary waves with $\omega$ sufficiently close
to $\omega=m$.
This shows that all solitary waves
of sufficiently small amplitude
are linearly unstable.
The approach is based
on applying the Rayleigh-Schr\"odinger perturbation
theory to the nonrelativistic limit of the equation.

The results are in formal agreement with the
Vakhitov-Kolokolov stability criterion.

Let us mention a similar independent result \cite{guan-gustafson-2010}
on linear instability
for the nonlinear Dirac equation in three dimensions,
with cubic nonlinearity
(this result is also in formal agreement
with the Vakhitov-Kolokolov stability criterion).
\end{abstract}

\section{Introduction}

A natural simplification
of the Dirac-Maxwell system
\cite{MR0190520}
is the nonlinear Dirac equation,
such as
the massive Thirring model \cite{MR0091788}
with vector-vector self-interaction
and
the Soler model \cite{PhysRevD.1.2766}
with scalar-scalar self-interaction
(known in dimension $n=1$
as the massive Gross-Neveu model
\cite{PhysRevD.10.3235,PhysRevD.12.3880}).
There was an enormous body of research
devoted to the nonlinear Dirac equation,
which we can not cover comprehensively in this short note.
The existence of standing waves in the nonlinear Dirac equation was
studied in
\cite{PhysRevD.1.2766},
\cite{MR847126}, \cite{MR949625}, and \cite{MR1344729}.
Numerical confirmation
of spectral stability
of solitary waves of small amplitude
is contained in \cite{chugunova-thesis}.
The overview of the well-posedness of the nonlinear
Dirac equation in 1D
is contained in \cite{2010arXiv1011.5925P}.
The asymptotic stability of small amplitude solitary waves
in the external potential
has been studied in \cite{MR2466169,2010arXiv1008.4514P}.

These models
with self-interaction of local type
have been
receiving a lot of attention in the
particle physics,
as well as in the theory of Bose-Einstein
condensates \cite{MR2542748,PhysRevLett.104.073603}.
The question of stability
of solitary waves
is of utmost importance:
perturbations ensure
that we only ever encounter stable configurations.
Recent attempts at asymptotic stability
\cite{2010arXiv1008.4514P,2011arXiv1103.4452B}
rely on the fundamental question
of \emph{spectral stability}:

\medskip
\noindent

\begin{verse}
{\it
Consider the Ansatz $\psi(x,t)=\big(\phi\sb\omega(x)+\rho(x,t)\big)e^{-i\omega t}$,
with $\phi\sb\omega(x)e^{-i\omega t}$ a solitary wave solution.

Let $\p\sb t\rho=A\sb\omega\rho$ be the linearized equation on $\rho$.
Does
$A\sb\omega$ have
eigenvalues
in the right half-plane?
}
\end{verse}

\medskip


In spite of a very clear picture
of the spectral stability of
nonlinear Schr\"odinger and Klein-Gordon equations
\cite{VaKo,MR723756,MR783974,MR804458,MR901236}
and 
many attempts at the spectral stability
in the context of the nonlinear Dirac
(let us mention
\cite{as1983,as1986,MR848095,MR2466169,PhysRevE.82.036604}),
in the latter case
the question
of spectral stability
is still completely open.
Our numerical results
\cite{dirac-1d-arxiv}
show that
in the 1D Soler model (cubic nonlinearity)
all solitary waves are spectrally stable.
Let us mention the related results
\cite{MR2217129,chugunova-thesis}.

Our previous result
\cite{dirac-vk-arxiv}
shows that the Vakhitov-Kolokolov criterion
in the case of the nonlinear Dirac equation
gives a less definite answer
about the spectral stability
than in the case of the nonlinear Schr\"odinger equation.
All we know is that
when
$\p\sb\omega Q(\omega)=0$,
with $Q(\omega)$ being the charge
of the solitary wave 
$\phi\sb\omega e^{-i\omega t}$,
two eigenvalues collide at $\lambda=0$,
but we do not know where these eigenvalues
are located when $\p\sb\omega Q(\omega)\ne 0$.
On the other hand,
for the solutions
to the nonlinear Schr\"odinger equation,
the condition $\p\sb\omega Q(\omega)>0$
is enough to conclude that
one positive and one negative eigenvalues
move out of $\lambda=0$
along the real axis;
see \cite{VaKo,MR901236,MR1995870}.



In the present paper, we show that
if the Vakhitov-Kolokolov guarantees linear
instability
for the system obtained in the nonrelativistic limit,
then the same result is true for the
small amplitude solitary waves in the original relativistic system.
We use our approach to show that the small amplitude
solitary wave solutions to the nonlinear Dirac equation
in 1D with higher order nonlinearities
are linearly unstable.

According to our results,
the spectrum of the linearization
at small amplitude solitary waves
in 1D nonlinear Dirac equation
with the cubic and quintic nonlinearities
has no real eigenvalues;
instead,
one can prove the existence of two purely imaginary eigenvalues.
To prove that these solitary waves are spectrally stable,
one also needs to prove that there are no complex
eigenvalues with $\Re\lambda>0$;
we already have partial results
which we will publish elsewhere.

\medskip

Our approach is based on the idea that
the family of real eigenvalues
of the linearization of the nonlinear Dirac equation
bifurcating from $\lambda=0$
is a deformed family of eigenvalues
of the linearization of the nonlinear Schr\"odinger equation.
The model and the main results
are described in Section~\ref{sect-results}.
The necessary constructions
in the context of the nonlinear Schr\"odinger equation
are presented in Section~\ref{sect-nls}.
The 
asymptotics of solitary waves
of the nonlinear Dirac equation
and the linearization 
is covered in Section~\ref{sect-linearization}.
The statement of Theorem~\ref{main-theorem-dirac-vbk}
for $\omega\lesssim m$
follows from Proposition~\ref{prop-k3},
which we prove using the Rayleigh-Schr\"odinger perturbation theory.

\medskip

\noindent
{ACKNOWLEDGMENTS.}
The author is grateful
to
Gregory Berkolaiko,
Nabile Boussa\"id,
Maria Esteban,
Dmitry Pelinovsky,
Iosif Polterovich,
Bj\"orn Sandstede,
Walter Strauss,
Boris Vainberg,
and Michael Weinstein
for most helpful discussions.
The author is grateful to
Stephen Gustafson
for the preprint \cite{guan-gustafson-2010}
with a similar result
for the nonlinear Dirac equation in 3D.

\section{Main result}
\label{sect-results}

We consider the nonlinear Dirac equation in
one dimension,
\begin{equation}\label{nld-1d}
i\p\sb t\psi=D\sb m\psi
-f(\psi\sp\ast\beta\psi)\beta\psi, \qquad x\in\R,
\qquad
\psi\in\mathbb{C}^2,
\end{equation}
where $D\sb m$ is the
Dirac operator:
\[
D\sb m=-i\alpha\p\sb{x}+m\beta,
\qquad
m>0.
\]
Above,
$\psi\sp\ast$
being the Hermitian conjugate of $\psi$.
We assume that the nonlinearity $f(s)$
is smooth, real-valued,
and satisfies $f(0)=0$.
The Hermitian matrices $\alpha$ and $\beta$
are chosen so that
\[
D\sb m^2=\big(-i\alpha\p\sb{x}+\beta m\big)^2=(-\p\sb x^2+m^2)I\sb{2},
\]
where $I\sb{2}$ is the $2\times 2$ unit matrix.
That is, $\alpha$ and $\beta$ are to satisfy
\begin{equation}
\label{eq:def-Dirac-matrices}
\alpha^2=I\sb{2},
\qquad
\beta^2=I\sb{2};
\qquad
\alpha\beta+\beta\alpha=0.
\end{equation}
The generalized massive Gross-Neveu model,
or, in the terminology of \cite{PhysRevE.82.036604},
the scalar-scalar case with $k\in\N$,
corresponds to the nonlinearity $f(s)=s^k$.

According to the Dirac-Pauli theorem
(cf. \cite{dirac-1928-616,vanderwaerden-1932,Pauli}
and \cite[Lemma 2.25]{thaller}),
the particular choice of $\alpha$ and $\beta$ matrices
does not matter.
We choose
\[
\alpha =-\sigma\sb 2 
= \left(\begin{matrix}0&i\\-i&0\end{matrix}\right),
\qquad
\beta=\sigma\sb 3
= \left(\begin{matrix}1&0\\0&-1\end{matrix}\right).
\]

For a large class of nonlinearities $f(s)$,
there are solitary wave solutions
of the form
\begin{equation}\label{sw}
\psi(x,t)=\phi\sb\omega(x)e^{-i\omega t},
\qquad
\phi\sb\omega
=\begin{bmatrix}v(\cdot,\omega)\\u(\cdot,\omega)\end{bmatrix}
\in H\sp 1(\R,\C^2),
\qquad
\abs{\omega}<m,
\end{equation}
with $v(x,\omega)$ and $u(x,\omega)$
real-valued
and with finite $H\sp 1$ norms.
For example, this is the case for the Soler model
with the nonlinearity $f(s)=s$.
For details, see Section~\ref{sect-solitary-waves-1d}.

Due to the $\mathbf{U}(1)$-invariance,
for solutions to \eqref{nld-1d}
the value of the charge functional
\[
Q(\psi)=\int\sb{\R}\abs{\psi(x,t)}^2\,dx
\]
is formally conserved.
For brevity,
we also denote by $Q(\omega)$
the charge of the solitary wave
$\phi\sb\omega(x)e^{-i\omega t}$:
\begin{equation}\label{def-q-omega}
Q(\omega)=\int\sb{\R}\abs{\phi\sb\omega(x)}^2\,dx.
\end{equation}

We are interested in the spectral stability
of linearization
of \eqref{nld-1d}
at a solitary wave solution \eqref{sw}.


\begin{theorem}\label{main-theorem-dirac-vbk}
Let
\[
f(s)=a s^k+O(s^{k+1}),
\qquad
a>0.
\]
If $k\ge 3$,
then
the solitary wave solutions
$\phi\sb\omega e^{-i\omega t}$
to \eqref{nld-1d}
are linearly unstable
for $\omega\in I$,
where $I\subset\Omega$ is the largest interval
with $\supp I=m$
such that $\p\sb\omega Q(\omega)$ does not vanish on $I$.
More precisely,
let $A\sb\omega$ be the linearization
of the nonlinear Dirac equation
at a solitary wave $\phi\sb\omega(x)e^{-i\omega t}$.
Then there are eigenvalues
$\pm\lambda\sb\omega\in\sigma\sb{p}(A\sb\omega)$,
with $\lambda\sb\omega>0$ for $\omega\in I$,
with $\lambda\sb\omega=O(m-\omega)$.
\end{theorem}

See Figure~\ref{fig-spectrum}.

\bigskip

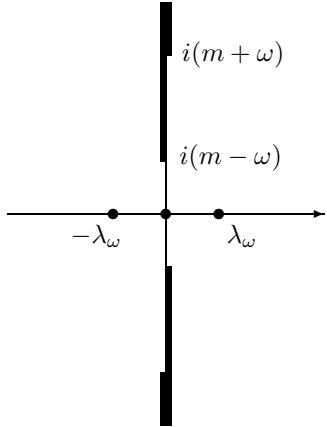
\begin{figure}[ht]
\begin{center}
\setlength{\unitlength}{1pt}
\begin{picture}(0,180)(0,-90)
\font\gnuplot=cmr10 at 10pt
\gnuplot

\put(6,58){$i(m+{\omega})$}
\put(5,19){$i(m-{\omega})$}

\put( 23,-11){$\lambda\sb\omega$}
\put(-36,-11){$-\lambda\sb\omega$}

\put(-20,0){\circle*{4}}
\put( 20,0){\circle*{4}}
\put(  0,0){\circle*{4}}

\put(-60,  0){\vector(1,0){120}}
\put(  0,-80){\vector(0,1){160}}

\linethickness{2pt}
\put( 1,-80){\line(0,1){60}}
\put(-1,-80){\line(0,1){20}}
\put( 1,80){\line(0,-1){20}}
\put(-1,80){\line(0,-1){60}}

\end{picture}
\end{center}
\caption{
\small
Main result:
The point spectrum of the linearization
of the nonlinear Dirac equation
with the nonlinearity
$f(s)=s^k+O(k^{k+1})$, $k\ge 3$,
at a solitary wave
with $\omega\lesssim m$
contains two nonzero real eigenvalues,
$\pm\lambda\sb\omega$,
with $\lambda\sb\omega=O(m-\omega)$.
See Theorem~\ref{main-theorem-dirac-vbk}.
Also plotted on this picture
is the essential spectrum,
with the edges
at $\lambda=\pm i(m-{\omega})$
and with the embedded threshold points
(branch points of the dispersive relation)
at $\lambda=\pm i(m+{\omega})$.
}

\label{fig-spectrum}

\end{figure}

\begin{remark}\label{remark-enough}
We only need to prove the linear instability
for $\omega\lesssim m$.
Then, by \cite{dirac-vk-arxiv},
the positive and negative eigenvalues
remain trapped on the real axis,
not being able
to collide at $\lambda=0$
for $\omega\in I=(\omega\sb 0,m)$
as long as $\p\sb\omega Q(\omega)$
does not vanish on $I$
(this nonvanishing is a sufficient condition
for the dimension of the generalized
null space to remain equal to two).
These eigenvalues can not leave into the complex
plane, either, since they are simple,
while the spectrum of the operator
is symmetric with respect to the real
and imaginary axes.
\end{remark}

\begin{remark}
We do not know what happens at $\omega=\inf I$:
as $\omega$ drops below $\inf I$,
it could be that either the pair of real
eigenvalues, having collided at $\lambda=0$,
turn into a pair of purely imaginary eigenvalues
(linear instability disappears),
or instead
two purely imaginary eigenvalues,
having met at $\lambda=0$,
turn into the
second  pair of real eigenvalues
(linear instability persists).
\end{remark}

\begin{remark}
Theorem~\ref{main-theorem-dirac-vbk}
is in the formal agreement with the Vakhitov-Kolokolov
stability criterion \cite{VaKo},
since
for $\omega\lesssim m$
one has $Q'(\omega)<0$ for $k=1$ 
and $Q'(\omega)>0$ for $k\ge 3$.
Let us mention that
the sign of the instability criterion,
$Q'(\omega)>0$
differs from \cite{VaKo}
because of their writing the solitary waves
in the form
$\varphi(x) e^{+i\omega t}$.
\end{remark}

\begin{remark}
We expect that in the 1D case with $k=1$
the small amplitude
solitary waves are spectrally stable;
we will prove this elsewhere.

We also expect that in the 1D case with $k=2$
(``quintic nonlinearity'',
$f(\psi\sp\ast\beta\psi)\beta\psi=O(\abs{\psi}^5)$)
the small solitary wave solutions
of the nonlinear Dirac equation in 1D
are spectrally stable.
For the corresponding nonlinear Schr\"odinger
equation (quintic nonlinearity in 1D),
the charge is constant,
thus the zero eigenvalue of a linearized operator
is always of higher algebraic multiplicity.
For the Dirac equation, this
degeneracy is ``resolved'':
the charge is now a decaying function
for $\omega\lesssim m$
(with nonzero limit as $\omega\to m$),
suggesting that
there are two purely imaginary eigenvalues
$\pm\lambda\sb\omega$
in the spectrum of $A\sb\omega$,
with $\lambda\sb\omega=o(m-\omega)$,
but no eigenvalues with nonzero real part.
\end{remark}

\begin{remark}
The same approach
can be used
to show instability of small amplitude solitary wave
solutions to the nonlinear Dirac equation in
$n\ge 3$, with
$f(s)=as^k+O(s^{k+1})$,
$a>0$, with any $k\in\N$.
These results in 3D have been
independently obtained in \cite{guan-gustafson-2010}.
Let us notice that
in the 3D case
for the cubic nonlinearity $f(s)=s$
(this is the original Soler model
from \cite{PhysRevD.1.2766}),
based on the numerical evidence from
\cite{PhysRevD.1.2766,as1983},
one expects that
the charge
$Q(\omega)$
has a local minimum at $\omega=0.936 m$,
suggesting that
the solitary waves
with $0.936 m<\omega<m$ are linearly unstable,
but then 
at $\omega=0.936 m$ the real eigenvalues
collide at $\lambda=0$,
and there are no nonzero real eigenvalues
in the spectrum
for $\omega\lesssim 0.936 m$.
\end{remark}

\begin{remark}
We can not rule out the possibility
that the eigenvalues with nonzero real part
could bifurcate directly from the imaginary axis
into the complex plane. Such a mechanism is absent
for the nonlinear Schr\"odinger equation
linearized at a solitary wave,
for which the point eigenvalues always remain
on the real or imaginary axes.
At present, though, we do not have examples of such
bifurcations
in the context of nonlinear Dirac equation
linearized at a solitary wave.
\end{remark}

\begin{remark}
The existence of solitary waves
stated in the theorem
follows from \cite{MR847126}.
We will reproduce their argument
in order to have the asymptotics
of profiles of solitary waves
for $\omega$ near $m$.
\end{remark}

\section{Nonlinear Schr\"odinger equation}
\label{sect-nls}

We are going to use the fact that
the nonlinear Dirac equation
in the nonrelativistic limit
coincides with the nonlinear Schr\"odinger
equation,
\begin{equation}\label{nls}
i\p\sb t\psi=-\frac 1 2\p\sb x^2\psi+m\psi-f(\abs{\psi}^2)\psi,
\qquad
\psi(x,t)\in\C,
\quad
x\in\R.
\end{equation}
We will assume that
\[
f(s)=s^{k},
\quad
k\in\N.
\]

\subsection*{Solitary waves}

The solitary wave solutions
\[
\psi(x,t)=\phi\sb\omega(x)e^{-i\omega t},
\qquad
\phi\sb\omega\in H\sp 1(\R),
\]
exist
for $\omega\in(-\infty,m)$.
These solitary waves and their
asymptotics are well-known, to the
extent that we do not know which of the
multiple references
would be most appropriate.
Often we will not indicate explicitly the dependence
of the amplitude of the solitary wave on $\omega$,
writing $\phi$ instead of $\phi\sb\omega$.
This amplitude satisfies the equation
\[
\p\sb x^2\phi(x)=2(m-\omega)\phi-2\phi^{2k+1},
\qquad
x\in\R,
\]
which could be integrated to the relation
\[
\p\sb x\phi
=-\phi\sqrt{2(m-\omega)-\frac{2\phi^{2k}}{k+1}}.
\]
Introducing $\mathscr{X}(x,\omega)=\phi\sb\omega^2(x)$,
we get
\begin{equation}\label{x-prime}
\p\sb x\mathscr{X}
=-2\mathscr{X}\sqrt{2(m-\omega)-\frac{2\mathscr{X}^{k}}{k+1}}.
\end{equation}
We will perform the scaling
in terms of 
\begin{equation}\label{varepsilon-omega}
\varepsilon=\sqrt{2(m-\omega)}.
\end{equation}
Then, as can be seen from \eqref{x-prime},
\begin{equation}\label{def-xk}
\mathscr{X}(x,\omega)=\varepsilon^{2/k}
U(\varepsilon x),
\end{equation}
with
$U(y)$
positive spherically symmetric (even)
solution to the equation
\[
\p\sb y U(y)
=-2U\sqrt{1-\frac{2U^k(y)}{k+1}},
\qquad
\lim\sb{y\to\pm\infty} U(y)=0.
\]
Such a solution exists and is unique;
it is explicitly given by
\begin{equation}\label{def-vk}
U(y)
=\Big(\frac{k+1}{2\cosh^2 k y}\Big)^{1/k}.
\end{equation}

\subsection*{Linearization at a solitary wave}

To derive the linearization
of the nonlinear Schr\"odinger equation \eqref{nls}
at a solitary wave
$\psi(x,t)=\phi\sb\omega(x)e^{-i\omega t}$,
we use the Ansatz
\[
\psi(x,t)=(\phi\sb\omega(x)+\rho(x,t))e^{-i\omega t},
\qquad
\rho(x,t)\in\C,
\qquad
x\in\R,
\]
and arrive at the linearized equation
\begin{equation}\label{nls-lin}
\p\sb t
\bm{\uprho}
=
\eub{j}\bm{L}(\omega)
\bm{\uprho},
\qquad
\bm{\uprho}(x,t)=\begin{bmatrix}\Re\rho(x,t)\\\Im\rho(x,t)\end{bmatrix},
\end{equation}
with
\begin{equation}\label{def-jl}
\eub{j}=
\begin{bmatrix}0&1\\-1&0\end{bmatrix},
\qquad
\bm{L}=
\begin{bmatrix}L\sb{-}&0\\0&L\sb{+}\end{bmatrix},
\end{equation}
\begin{equation}
L\sb{-}(\omega)=-\frac{1}{2}\p\sb x^2+m-\omega-f(\phi\sb\omega^2),
\qquad
L\sb{+}(\omega)=L\sb{-}-2f'(\phi\sb\omega^2)\phi\sb\omega^2.
\end{equation}
In the case of the nonlinearity
$f(s)=s^k$,
taking into account
the explicit form of $\mathscr{X}(x)$
given by 
\eqref{def-xk},
one obtains
\begin{equation}\label{l0-l1}
L\sb{-}(\omega)=
-\frac 1 2\p\sb x^2
+m-\omega-\varepsilon^2\frac{k+1}{2\cosh^2(\varepsilon k x)},
\qquad
L\sb{+}(\omega)=
-\frac 1 2\p\sb x^2
+m-\omega-\varepsilon^2\frac{(2k+1)(k+1)}{2\cosh^2(\varepsilon k x)}.
\end{equation}

\begin{lemma}[Vakhitov-Kolokolov stability criterion]
\label{lemma-vk}
For the linearization
\eqref{nls-lin}
at a particular solitary wave
$\phi\sb\omega(x)e^{-i\omega t}$,
there are real nonzero eigenvalues
$\pm\lambda\in\sigma\sb{d}(\eub{j}\bm{L})$,
$\lambda>0$,
if and only if
$\frac{d}{d\omega}\norm{\phi\sb\omega}\sb{L\sp 2}^2>0$
at this value of $\omega$.
\end{lemma}


Let us consider the spectrum of
the linearized equation \eqref{nls-lin}
for the nonlinearity $f(s)=s^k$.
Using \eqref{def-xk},
one derives 
for the corresponding charge:
\begin{equation}\label{qp}
Q(\omega)
=\int\sb{\R}\phi\sb\omega^2(x)\,dx
=\int\sb{\R}\mathscr{X}(x,\omega)\,dx
=\varepsilon^{\frac 2 k -1}C
=(2(m-\omega))^{\frac 1 k -\frac 1 2}C,
\qquad
\omega<m,
\end{equation}
where
$C=\int\sb{\R}U(y)\,dy>0$.
We see from \eqref{qp}
that
one has
$Q'(\omega)<0$ for $k=1$,
$Q'(\omega)\equiv 0$ for $k=2$,
and $Q'(\omega)>0$ for $k\ge 3$.

\begin{lemma}\label{lemma-nls-k}
Let
$f(s)=s^k
$,
$k\in\N$.
If
$k=1$ or $k=2$,
then
$\sigma\sb{p}(\eub{j}\bm{L})\subset i\R$.
If $k\ge 3$,
then
$\sigma\sb{p}(\eub{j}\bm{L})\ni\{\pm\varepsilon^2\Lambda\}$,
for some $\Lambda>0$
and
$\varepsilon=\sqrt{2(m-\omega)}$.
\end{lemma}

\begin{proof}
In the case $k=1$,
since $Q'(\omega)<0$ by \eqref{qp},
Lemma~\ref{lemma-vk}
guarantees that there are no nonzero real eigenvalues:
$\sigma\sb{p}(\eub{j}\bm{L}(\omega))\cap(\R\backslash 0)=\emptyset$.
In the case $k=2$, $Q'(\omega)\equiv 0$,
and the eigenvalue $\lambda=0$
is always of increased algebraic multiplicity
$6$
(generically,
the algebraic multiplicity of $\lambda=0$
is $4$,
jumping to $6$ when $Q'(\omega)=0$).
In the case $k\ge 3$,
since $Q'(\omega)>0$ by \eqref{qp},
Lemma~\ref{lemma-vk} states that
there are two real eigenvalues
$\pm\lambda\sb\omega\in\sigma\sb{p}(\eub{j}\bm{L}(\omega))$,
with $\lambda\sb\omega>0$.
The rescaling $y=\varepsilon x$
shows that $\lambda\sb\omega=\varepsilon^2\Lambda$,
with some $\Lambda>0$.
\end{proof}



For $k\in\N$, consider
the Schr\"odinger operators
\begin{equation}\label{l-hat-1}
\hat{\bm{L}}
=\begin{bmatrix}\hat{L}\sb{-}&0\\0&\hat{L}\sb{+}\end{bmatrix},
\qquad
\hat{L}\sb{-}
=
-\frac 1 2 \p\sb y^2+\frac 1 2-\frac{k+1}{2\cosh^2 ky},
\qquad
\hat{L}\sb{+}
=
-\frac 1 2 \p\sb y^2
+
\frac 1 2-\frac{(2k+1)(k+1)}{2\cosh^2 ky}.
\end{equation}
Let $\upphi(y)=\frac{1}{\cosh^{1/k}ky}$; then
\begin{equation}\label{h0-varphi}
\hat{L}\sb{-} \upphi=0,
\qquad
\hat{L}\sb{+} \p\sb x\upphi=0.
\end{equation}

\begin{proposition}\label{prop-a-hat}
If $k\ge 3$,
the discrete spectrum of the operator
$\eub{j}\hat{\bm{L}}
=\begin{bmatrix}0&\hat{L}\sb{-}\\-\hat{L}\sb{+}&0\end{bmatrix}$
contains two real nonzero eigenvalues
$\pm\Lambda$,
$\Lambda>0$.
If $k\le 2$,
$\sigma\sb d(\eub{j}\hat{\bm{L}})\subset i\R$.
\end{proposition}

\begin{proof}
We start with constructing several important
relations.
By \eqref{h0-varphi},
\begin{equation}\label{l1-varphi}
\hat{L}\sb{+} \upphi=
\hat{L}\sb{-} \upphi+(\hat{L}\sb{+}-\hat{L}\sb{-})\upphi
=-\frac{k(k+1)}{\cosh^2 ky}\upphi.
\end{equation}
For $\mu>0$,
let
$\upphi\sb\mu(y)=\upphi(\mu y)$
and
$
\hat{L}\sb{\mu}
=
-\frac 1 2 \p\sb y^2+\mu^2
\left(\frac 1 2-\frac{k+1}{2\cosh^2 \mu ky}\right);
$
then
$\hat{L}\sb{\mu}\upphi\sb\mu=0$.
Taking the derivative
of this relation
with respect to $\mu$
and evaluating it at $\mu=1$, we get:
\begin{equation}\label{l1-theta}
\hat{L}\sb{+}\uptheta
=
\Big(
-1+\frac{k+1}{\cosh^2 ky}\Big)\upphi,
\qquad
\mbox{where}
\quad
\uptheta:=\frac{\p}{\p\mu}\At{\mu=1}
\frac{1}{\cosh^{\frac{1}{k}}\mu ky}
=-\frac{y\sinh ky}{\cosh^{1+\frac{1}{k}}ky}.
\end{equation}
By \eqref{l1-varphi} and \eqref{l1-theta},
\begin{equation}\label{l1-is-varphi}
\hat{L}\sb{+}(-\uptheta-\frac{1}{k}\upphi)=\upphi.
\end{equation}

Now we follow the argument from \cite{VaKo}.

\begin{lemma}\label{lemma-a2}
Let $k\ge 3$.
The minimum of
$\langle r,\hat{L}\sb{+} r\rangle$
under the constraints
$\langle r,r\rangle=1$ and $\langle\upphi,r\rangle=0$
is negative.
\end{lemma}

\begin{proof}
The vector $r$
corresponding to the minimum of $\langle r,\hat{L}\sb{+} r\rangle$
under the constraints
$\langle r,r\rangle=1$ and $\langle\upphi,r\rangle=0$
satisfies the equation
$
\hat{L}\sb{+} r=\alpha r+\beta\upphi,
$
where $\alpha,\,\beta\in\R$ are Lagrange multipliers.
Since
$\langle r,\hat{L}\sb{+} r\rangle=\langle r,\alpha r+\beta\upphi\rangle=\alpha$,
we need to know the sign of $\alpha$.
Denote
\[
f(z)=\langle\upphi,(\hat{L}\sb{+}-z)^{-1}\upphi\rangle,
\qquad
z\in\rho(\hat{L}\sb{+}).
\]
We note that
$f(z)$ has a removable singularity at $z=0$
since
$\upphi\in L\sp 2\sb r(\R)$,
while the restruction of $\hat{L}\sb{+}$
onto the space of spherically symmetric (that is, even) functions,
$\hat{L}\sb{+}:L\sp 2\sb r(\R)\to L\sp 2\sb r(\R)$,
has a bounded inverse.
 (being even)
Using \eqref{l1-is-varphi},
we have:
\[
f(0)
=\langle\upphi,\hat{L}\sb{+}^{-1}\upphi\rangle
=\langle\upphi,(-\uptheta-\frac{1}{k}\upphi)\rangle
=
\int\sb{\R}
\frac{1}{\cosh^{\frac{1}{k}}ky}
\left(
\frac{y\sinh ky}{\cosh^{1+\frac{1}{k}}ky}
-\frac{1}{k\cosh^{\frac{1}{k}}ky}
\right)\,dy
\]
\[
=
\int\sb{\R}
\Big(
\frac{z\sinh z}{\cosh^{1+\frac{2}{k}}z}
-\frac{dz}{\cosh^{\frac{2}{k}}z}
\Big)
\frac{dz}{k}
=\int\sb{\R}
\Big(
-z\,d\Big(\frac{1}{2\cosh^{\frac{2}{k}}z}\Big)
-\frac{dz}{k\cosh^{\frac{2}{k}}z}
\Big)
=
\big(\frac{1}{2}-\frac{1}{k}\big)\int\sb{\R}
\frac{dz}{\cosh^{\frac{2}{k}}z}>0,
\]
where we took into account that $k\ge 3$.
Since $f(0)>0$
and $\lim\limits\sb{z\to\lambda\sb 0+}f(z)=-\infty$
(it is $-\infty$ since $f'(z)>0$),
where $\lambda\sb 0<0$ is the smallest eigenvalue of
$\hat{L}\sb{+}$,
there is $\mu\in(\lambda\sb 0,0)$ such that $f(\mu)=0$.
\end{proof}

Since
$\hat{L}\sb{-}$ is positive-definite
and $r$ in Lemma~\ref{lemma-a2}
is orthogonal to $\ker \hat{L}\sb{-}$
(which is spanned by $\upphi$),
we may define $R=\hat{L}\sb{-}^{-1/2}r$;
then
\[
\langle R,\hat{L}\sb{-}^{1/2}\hat{L}\sb{+} \hat{L}\sb{-}^{1/2}R\rangle
=\langle r,\hat{L}\sb{+} r\rangle<0.
\]
Therefore,
$\sigma\sb d(\hat{L}\sb{-} \hat{L}\sb{+})
=\sigma\sb d(\hat{L}\sb{-}^{1/2}\hat{L}\sb{+} \hat{L}\sb{-}^{1/2})$
contains a negative eigenvalue $-\Lambda^2$,
where $\Lambda>0$.
Let $\xi$ be a corresponding eigenvector,
so that
$\hat{L}\sb{-} \hat{L}\sb{+}\xi=-\Lambda^2\xi$.
Then
\[
\begin{bmatrix}0&\hat{L}\sb{-}\\-\hat{L}\sb{+}&0\end{bmatrix}
\begin{bmatrix}\Lambda\xi\\\mp \hat{L}\sb{+}\xi\end{bmatrix}
=
\pm\Lambda
\begin{bmatrix}\Lambda\xi\\\mp \hat{L}\sb{+}\xi\end{bmatrix},
\qquad
\mbox{hence}
\quad
\pm\Lambda\in\sigma\sb{d}(\eub{j}\hat{\bm{L}}).
\]
\end{proof}

\begin{remark}
Since the spectrum of the linearization
at zero solitary wave is purely imaginary,
one has $\lim\limits\sb{\omega\to m-}\lambda\sb\omega=0$.
\end{remark}

\begin{remark}
Comparing $\bm{L}$
(see \eqref{def-jl}, \eqref{l0-l1})
to the operator $\hat{\bm{L}}$
introduced in \eqref{l-hat-1},
\begin{equation}\label{l-hat-2}
\hat{\bm{L}}
=\begin{bmatrix}\hat{L}\sb{-}&0\\0&\hat{L}\sb{+}\end{bmatrix},
\qquad
\hat{L}\sb{-}
=
-\frac 1 2 \p\sb y^2
+\frac 1 2-\frac{k+1}{2\cosh^2 k y},
\qquad
\hat{L}\sb{+}
=
-\frac 1 2 \p\sb y^2
+\frac 1 2-\frac{(2k+1)(k+1)}{2\cosh^2 k y},
\end{equation}
one concludes that
\[
\sigma(L\sb{-})=\varepsilon^2\sigma(\hat{L}\sb{-}),
\qquad
\sigma(L\sb{+})=\varepsilon^2\sigma(\hat{L}\sb{+}),
\qquad
\sigma(\eub{j}\bm{L})=\varepsilon^2\sigma(\eub{j}\hat{\bm{L}}).
\]
Therefore,
$\lambda\sb\omega=\varepsilon^2\Lambda$,
where $\Lambda$ is the positive eigenvalue of
$\eub{j}\hat{\bm{L}}$.
\end{remark}


\section{Nonlinear Dirac equation}
\label{sect-linearization}

In this section, we will use the notation
\[
g(s)=m s-f(s).
\]
In terms of
the components,
$\psi=\begin{bmatrix}\psi\sb 1\\\psi\sb 2\end{bmatrix}$,
$\psi\sb 1$, $\psi\sb 2\in\C$,
we write the nonlinear Dirac equation \eqref{nld-1d} as a system
\begin{equation}\label{nld-2c}
\left\{
\begin{array}{ll}
i\p\sb t \psi\sb 1=\p\sb x \psi\sb 2
+g(\abs{\psi\sb 1}^2-\abs{\psi\sb 2}^2)\psi\sb 1,
\\
i\p\sb t \psi\sb 2=-\p\sb x \psi\sb 1
-g(\abs{\psi\sb 1}^2-\abs{\psi\sb 2}^2)\psi\sb 2.
\end{array}
\right.
\end{equation}

\subsection*{Solitary waves}
\label{sect-solitary-waves-1d}

\begin{definition}
The solitary waves are solutions
to \eqref{nld-1d}
of the form
\[
\mathbf{S}=\{
\psi(x,t)=\phi\sb\omega(x)e^{-i\omega t}\sothat
\phi\sb\omega\in H\sp 1(\R,\C^2),\ \omega\in\R\}.
\]
\end{definition}

We start by demonstrating the existence of solitary wave solutions and
exploring their properties.
The following result follows from \cite{MR847126};
we follow our article \cite{dirac-1d-arxiv}.

\begin{lemma}\label{lemma-existence-nld-1d}
Let $G$ be the antiderivative of $g$ such that $G(0)=0$.
Assume that
there is $\omega\sb 0<m$
such that
for given $\omega\in(\omega\sb 0,m)$
there exists
$\varGamma\sb\omega>0$ such that
\begin{equation}\label{def-Xi}
\omega\varGamma\sb\omega=G(\varGamma\sb\omega), \qquad
\omega\ne g(\varGamma\sb\omega),
\qquad\mbox{and}\qquad
\omega s<G(s)\quad{\rm for}
\ s\in(0,\varGamma\sb\omega).
\end{equation}
Then there is a solitary wave solution
$\psi(x,t)=\phi\sb\omega(x)e^{-i\omega t}$
to \eqref{nld-1d},
where
\begin{equation}\label{psi-v-u}
\phi\sb\omega(x)
=\begin{bmatrix}v(x,\omega)\\u(x,\omega)\end{bmatrix},
\end{equation}
with both $v$ and $u$ real-valued,
belonging to $H\sp 1(\R)$ as functions of $x$,
$v$ being even and $u$ odd.

More precisely,
for $x\in\R$
and $\omega\in(\omega\sb 0,m)$,
let us define $\mathscr{X}(x,\omega)$
and $\mathscr{Y}(x,\omega)$ by
\begin{equation}\label{def-xi-eta}
\mathscr{X}=v^2-u^2,\qquad \mathscr{Y}=vu.
\end{equation}
Then $\mathscr{X}(x,\omega)$ is a unique positive
symmetric solution to
\begin{equation}
\p\sb x^2\mathscr{X}
=-\p\sb\mathscr{X}(-2G(\mathscr{X})^2+2\omega^2\mathscr{X}^2),
\qquad
\lim\sb{x\to\pm\infty}\mathscr{X}(x,\omega)=0,
\end{equation}
and $\mathscr{Y}(x,\omega)=-\frac{1}{4\omega}\p\sb x\mathscr{X}(x,\omega)$.
This solution satisfies
$\mathscr{X}(0,\omega)=\varGamma\sb\omega$.
\end{lemma}

\begin{proof}
From (\ref{nld-2c}), we obtain:
\begin{equation}\label{omega-v-u}
\left\{
\begin{array}{ll}
\omega v=\p\sb x u+g(\abs{v}^2-\abs{u}^2)v,
\\
\omega u=-\p\sb x v-g(\abs{v}^2-\abs{u}^2)u.
\end{array}
\right.
\end{equation}
Assuming that both $v$ and $u$ are real-valued
(this will be justified once we found
real-valued $v$ and $u$),
we can rewrite (\ref{omega-v-u}) as the following
Hamiltonian system:
\begin{equation}\label{stat-eqn}
\left\{
\begin{array}{ll}
\p\sb x u=\omega v-g(v^2-u^2)v=\p\sb{v} h(v,u),
\\
-\p\sb x v=\omega u+g(v^2-u^2)u=\p\sb{u} h(v,u),
\end{array}
\right.
\end{equation}
where the Hamiltonian $h(v,u)$ is given by
\begin{equation}\label{def-h}
h(v,u)=\frac{\omega}{2}(v^2+u^2)-\frac{1}{2}G(v^2-u^2).
\end{equation}
The solitary wave
with a particular $\omega\in(\omega\sb 0,m)$
corresponds to a trajectory of this Hamiltonian
system such that
\[
\lim\sb{x\to\pm\infty}v(x,\omega)
=\lim\sb{x\to\pm\infty}u(x,\omega)=0,
\]
hence
$\lim\sb{x\to\pm\infty}\mathscr{X}=0$.
Since $G(s)$ satisfies
$G(0)=0$,
we conclude that
$
h(v(x),u(x))\equiv 0,
$
which leads to
\begin{equation}\label{phi-phi-g}
\omega(v^2+u^2)=G(v^2-u^2).
\end{equation}
We conclude from
\eqref{phi-phi-g}
that
solitary waves may only correspond to $\abs{\omega}<m$,
$\omega\ne 0$.


The functions $\mathscr{X}(x,\omega)$ and $\mathscr{Y}(x,\omega)$
introduced in
\eqref{def-xi-eta} are to solve
\begin{equation}\label{xi-eta-system}
\left\{
\begin{array}{ll}
\p\sb x\mathscr{X}=-4\omega\mathscr{Y},
\\
\p\sb x\mathscr{Y}=-(v^2+u^2)g(\mathscr{X})+\omega\mathscr{X}
=-\frac{1}{\omega}G(\mathscr{X})g(\mathscr{X})+\omega\mathscr{X},
\end{array}
\right.
\end{equation}
and to have the asymptotic behavior
$\lim\sb{\abs{x}\to\infty}\mathscr{X}(x)=0$,
$\lim\sb{\abs{x}\to\infty}\mathscr{Y}(x)=0$.
In the second equation in
(\ref{xi-eta-system}), we used the relation (\ref{phi-phi-g}).
The system (\ref{xi-eta-system}) can be
written as the following equation on $\mathscr{X}$:
\begin{equation}\label{xi-p-p}
\p\sb x^2\mathscr{X}
=-\p\sb\mathscr{X}(-2 G(\mathscr{X})^2+2\omega^2\mathscr{X}^2)
=4
\big(G(\mathscr{X})g(\mathscr{X})-\omega^2\mathscr{X}\big).
\end{equation}
This equation describes a particle in the potential
$-2 G(s)^2+2\omega^2 s^2$.
The condition \eqref{def-Xi}
is needed for the existence
of the turning point
of the zero energy trajectory
in this potential,
at $s=\varGamma\sb\omega$.
\end{proof}

\subsection*{Solitary waves in the nonrelativistic limit}

From now on,
without loss of generality,
we will assume that
the nonlinearity
$f\in C\sp\infty(\R)$
is such that
\[
f(s)=s^k+o(s^{k+1}),
\qquad
k\in\N,
\]
and consider $k$ fixed throughout
the rest of the article
(this is to avoid writing subscripts).
By Lemma~\ref{lemma-existence-nld-1d},
there is $\omega\sb 0<m$
such that
there are solitary wave solutions
for $\omega\in(\omega\sb 0,m)$.

We are going to determine
the asymptotics of solitary waves
in the nonrelativistic limit $\omega\to m$.

\bigskip

We consider the function $g(s):=m-f(s)$ and its antiderivative,
\[
g(s)=m-s^k+O(s^{k+1}),
\qquad
G(s)=ms-\frac{s^{k+1}}{k+1}+O(s^{k+2}).
\]
Then
the expression in right-hand side of \eqref{xi-p-p}
takes the following form:
\[
G(s)
g(s)
-\omega^2 s
=
\big(ms-\frac{s^{k+1}}{k+1}+O(s^{k+2})\big)
\big(m-s^k+O(s^{k+2})\big)
-\omega^2s
=(m^2-\omega^2)s
-\frac{k+2}{k+1}s^{k+1}
+ O(s^{k+2}).
\]
Let $\omega\in(\omega\sb 0,m)$,
so that there is a solitary wave
$\phi\sb\omega(x)e^{-i\omega t}$
with this particular value of $\omega$.
Denote
\begin{equation}\label{epsilon-omega}
\epsilon=\sqrt{m^2-\omega^2}.
\end{equation}
Let
$\mathscr{X}(x,\omega)$
be a positive symmetric solution to
\eqref{xi-p-p}
from Lemma~\ref{lemma-existence-nld-1d};
we rewrite \eqref{xi-p-p} as
\begin{equation}\label{xpp}
-\frac{1}{4}
\p\sb x^2\mathscr{X}
-\frac{k+2}{k+1}\mathscr{X}^{k+1}
+\epsilon^2\mathscr{X}
=O(\mathscr{X}^{k+2}).
\end{equation}

Let $U(y)$ be a positive symmetric solution to
\begin{equation}\label{fpp}
-\frac{1}{4}\p\sb y^2 U-\frac{k+2}{k+1}U^{k+1}+U=0,
\qquad
\lim\sb{y\to\infty}U(y)=0.
\end{equation}
Such a solution $U(y)$ exists and is unique,
and is explicitly given by \eqref{def-vk}.

\begin{lemma}\label{lemma-xx-ff}
For $\omega\in(\omega\sb 0,m)$,
there are the relations
\begin{equation}\label{xx-ff}
\mathscr{X}(x,\omega)=
\epsilon^{\frac 2 k}U(\epsilon x)+O(\epsilon^{\frac 4 k}),
\end{equation}
\begin{equation}\label{uv-ff}
v(x,\omega)=\epsilon^{\frac{1}{k}}
(U(\epsilon x))^{\frac 1 2}
+O(\epsilon^{\frac{3}{k}}),
\qquad
u(x,\omega)=O(\epsilon^{1+\frac{1}{k}}).
\end{equation}
\end{lemma}

\begin{remark}
Similar asymptotics are established
in \cite{2008arXiv0812.2273G}
for the nonlinear Dirac equation in 3D.
These asymptotics are used in \cite{guan-gustafson-2010}
for the proof of linear instability
of solitary waves with $\omega\lesssim m$
to the nonlinear Dirac equation
with cubic nonlinearity.
\end{remark}

\begin{proof}
Since $U(y)>0$ for $y\in\R$, one concludes
from \eqref{fpp}
that $\lambda=0$
is the lowest eigenvalue of the operator
\[
\eur{H}\sb{-}
=
-\frac{1}{4}\p\sb y^2-\frac{k+2}{k+1}U^{k}+1.
\]
Taking the derivative of \eqref{fpp}, we find:
\begin{equation}\label{h1pf}
\eur{H}\sb{+} \p\sb y U
:=-\frac 1 4\p\sb y^3U-(k+2)U^{k}\p\sb y U+\p\sb y U
=0.
\end{equation}
The operator $\eur{H}\sb{+}$
in this relation
is the same as in \eqref{l-hat-1}.
With $\p\sb y U$ having one node,
we conclude that $\lambda=0$
is the second lowest eigenvalue of $\eur{H}\sb{+}$.

We define the function $X(y,\epsilon)$ by the relation
$\mathscr{X}(x,\omega)=\epsilon^{2/k}X(\epsilon x,\epsilon)$,
where $\epsilon$ and $\omega$ are related by
\eqref{epsilon-omega}.
Then \eqref{xpp} takes the form
\begin{equation}\label{xxpp}
-\frac 1 4\p\sb y^2 X+X-\frac{k+2}{k+1}
X^{k+1}
=\epsilon^{-2-\frac 2 k}O(\mathscr{X}^{k+2})
=\epsilon^{2/k}O(X^{k+2}).
\end{equation}
Subtracting \eqref{fpp} from \eqref{xxpp},
we find:
\[
-\frac 1 4
\p\sb y^2(X-U)+(X-U)
-\frac{k+2}{k+1}(X^{k+1}-U^{k+1})=\epsilon^{2/k}O(X^{k+2}),
\]
which we rewrite as
\[
-\frac 1 4
\p\sb y^2(X-U)+(X-U)
-(k+2)U^{k}(X-U)
=\epsilon^{2/k}O(X^{k+2})+O((X-U)^2).
\]
Thus,
denoting $Z(y,\epsilon)=X(y,\epsilon)-U(y)$,
one has
\begin{equation}\label{z-is-z}
Z
=\big(\eur{H}\sb{+}\at{L\sp 2\sb r}\big)^{-1}
\Big(
\epsilon^{2/k}O(U^{k+2})+\epsilon^{2/k}O(Z)+O(Z^2)
\Big).
\end{equation}
Above, $\eur{H}\sb{+}\at{L\sp 2\sb r}:L\sb r\sp 2(\R)\to L\sb r\sp 2(\R)$
is the restriction of $\eur{H}\sb{+}$ to the space
of spherically symmetric (even) functions;
its inverse is bounded
(from $L\sp 2\sb r$ to $H\sp 1\sb r$)
since by \eqref{h1pf}
we know that
$\ker \eur{H}\sb{+}$ is spanned by $\p\sb y U$,
which is skew-symmetric
(let us mention
that $\lambda=0$ is a simple eigenvalue of $\eur{H}\sb{+}$,
which is straightforward in one dimensional case).
The right-hand side of \eqref{z-is-z}
is well-defined since both $U$ and $Z=X-U$
are spherically symmetric.

For given $\epsilon>0$,
let
$\mathcal{B}\sb{\epsilon^{1/k}}
=\{Z\in H^1(\R)
\sothat Z(x)=Z(-x),\ \norm{Z}\sb{H\sp 1}<\epsilon^{1/k}\}
\subset H\sb r^1(\R)$
be the space
of symmetric functions with $\norm{Z}\sb{H\sp 1}\le\epsilon^{1/k}$.
Consider the map
\begin{equation}\label{z-to-z}
\mathcal{B}\sb{\epsilon^{1/k}}\to H\sp 1\sb r(\R),
\qquad
Z\mapsto
\big(\eur{H}\sb{+}\at{L\sp 2\sb r}\big)^{-1}
\Big(
\epsilon^{2/k}O(U^{k+2})+\epsilon^{2/k}O(Z)+O(Z^2)
\Big).
\end{equation}
We can choose
$\omega\sb 1\in(\omega\sb 0,m)$ so close to $m$,
that
for $\omega\in(\omega\sb 1,m)$
the value of
$\epsilon=\sqrt{m^2-\omega^2}$ is small enough
for the map \eqref{z-to-z} to be an endomorphism and a contraction
on $\mathcal{B}\sb{\epsilon^{1/k}}$;
by the contraction mapping theorem,
there is a unique stationary point
$Z\in\mathcal{B}\sb{\epsilon^{1/k}}$,
and, by \eqref{z-is-z},
$Z$
satisfies $\norm{Z}\sb{H\sp 1}=O(\epsilon^{2/k})$.
It follows that
\[
\norm{\mathscr{X}(x,\omega)-\epsilon^{2/k}U(\epsilon x)}\sb{L\sp\infty}
=\epsilon^{2/k}\norm{Z}\sb{L\sp\infty(\R)}
=O(\epsilon^{4/k}).
\]
Now the asymptotics of $\mathscr{Y}(x,\omega)$
could be determined from \eqref{xi-eta-system}
and then $v(x,\omega)$, $u(x,\omega)$
are determined by \eqref{def-xi-eta}.
\end{proof}

\section{Linear instability of small solitary waves}

\subsection*{Linearization at a solitary wave}

To derive the linearization of equation \eqref{nld-1d}
at a solitary wave \eqref{sw},
we consider the solution in the form of the Ansatz
\begin{equation}
\psi(x,t)=(\phi\sb\omega(x)+\rho(x,t))e^{-i\omega t},
\qquad
\phi\sb\omega(x)
=\begin{bmatrix}v(x,\omega)\\u(x,\omega)\end{bmatrix},
\qquad
\phi\sb\omega\in H\sp 1(\R,\C^2),
\end{equation}
where $\rho(x,t)\in\C^2$.
Note that
$v(x,\omega)$ and $u(x,\omega)$ are real-valued
by Lemma~\ref{lemma-existence-nld-1d}.
Then, by \eqref{nld-1d},
the linearized equation on $\rho$
takes the following form:
\begin{equation}\label{nld-1d-lin}
i\dot{\rho} =D\sb m\rho
-\omega\rho
-
f(\phi\sb\omega\sp\ast\beta\phi\sb\omega)\beta\rho
-(\phi\sb\omega\sp\ast\beta\rho +\rho\sp\ast\beta\phi\sb\omega)
f'(\phi\sb\omega\sp\ast\beta\phi\sb\omega)\beta\phi\sb\omega,
\qquad
D\sb m=-i\alpha\p\sb x+m\beta.
\end{equation}
We note that the above equation is $\R$-linear but not $\C$-linear, due
to the presence of the $\rho\sp\ast$ term.
We denote
$
\rho=\begin{bmatrix}
R\sb 1+iS\sb 1
\\
R\sb 2+iS\sb 2
\end{bmatrix},
$
with $R\sb{j}$, $S\sb{j}$ real-valued.
Since $v$ and $u$ are real-valued,
\[
\phi\sb\omega\sp\ast\beta\rho
+\rho\sp\ast\beta\phi\sb\omega
=2\Re
(\phi\sb\omega\sp\ast\beta\rho)
=2(v R\sb 1-u R\sb 2),
\]
we rewrite equation \eqref{nld-1d-lin}
in the following form
in terms of $R=\begin{bmatrix}R\sb 1\\R\sb 2\end{bmatrix}$,
$S=\begin{bmatrix}S\sb 1\\S\sb 2\end{bmatrix}$:
\begin{equation}\label{nld-1d-lin-2}
\p\sb t
\begin{bmatrix}R\\S\end{bmatrix}
= \eubJ\left\{
\begin{bmatrix}
D\sb m-\omega-f\beta & 0\\
0 &D\sb m-\omega-f\beta
\end{bmatrix}
\begin{bmatrix}R\\S\end{bmatrix}
-2\Re(\phi\sb\omega\sp\ast\beta\rho)
f'
\begin{bmatrix}\beta\phi\sb\omega\\0\end{bmatrix}
\right\},
\end{equation}
where
\[
f=f(\phi\sb\omega\sp\ast\beta\phi\sb\omega),
\qquad
f'=f'(\phi\sb\omega\sp\ast\beta\phi\sb\omega),
\]
and $\eubJ$ corresponds to $1/i$:
\[
\eubJ=
\begin{bmatrix}
0&I\sb 2\\-I\sb 2&0
\end{bmatrix},
\]
where $I\sb 2$ is the $2\times 2$ unit matrix.
We define $\eurL\sb{\pm}(\omega)$
and $\eubL(\omega)$
by the following:
\begin{equation}\label{def-dpm}
\eurL\sb{-}(\omega)
=
\begin{bmatrix}
m-\omega-f&\p\sb x
\\
-\p\sb x&-m-\omega+f
\end{bmatrix},
\qquad
\eurL\sb{+}(\omega)
=
\begin{bmatrix}
m-\omega-f-2f'v^2
&\p\sb x+2f'vu
\\
-\p\sb x+2f'vu&-m-\omega+f-2f'u^2
\end{bmatrix};
\end{equation}
\[
\eubL(\omega)
=\begin{bmatrix}\eurL\sb{+}(\omega)&0\\0&\eurL\sb{-}(\omega)\end{bmatrix}.
\]
Let us remind the reader that both $v$ and $u$
in \eqref{def-dpm}
depend on $\omega$.
Then equation \eqref{nld-1d-lin-2}
which describes the linearization at the solitary wave
$\phi\sb\omega e^{-i\omega t}$
takes the form
\[
\p\sb t
\begin{bmatrix}R\\S\end{bmatrix}
=
\eubJ\eubL(\omega)
\begin{bmatrix}R\\S\end{bmatrix}
=
\begin{bmatrix}0&\eurL\sb{-}(\omega)\\-\eurL\sb{+}(\omega)&0\end{bmatrix}
\begin{bmatrix}R\\S\end{bmatrix}.
\]

\begin{lemma}\label{lemma-hpm-2omega}
For any nonlinearity
$f(s)$
in \eqref{nld-1d},
the spectrum of the linearization at a solitary
wave
$\phi\sb\omega e^{-i\omega t}$
has the following properties:
\begin{enumerate}
\item
$
\sigma\sb{ess}(\eubJ\eubL(\omega))
=
i\R\backslash(-i(m-\omega),i(m-\omega));
$
\item
$\spec\sb{d}(\eubJ\eubL(\omega))\supset\{\pm 2\omega i;\,0\}$.
\end{enumerate}
\end{lemma}

See \cite{dirac-vk-arxiv}.



\subsection*{Unstable eigenvalue of $\eubJ\eubL$ for $\omega\lesssim m$}

\begin{proposition}\label{prop-k3}
Let $k\ge 3$.
There is $\omega\sb 0<m$
such that
for $\omega\in(\omega\sb 0,m)$
there are two families of eigenvalues
\[
\pm\lambda\sb\omega\in\sigma\sb p(\eubJ\eubL(\omega)),
\qquad
\lambda\sb\omega>0,
\qquad
\lambda\sb\omega=O(\epsilon^2).
\]
\end{proposition}

\begin{proof}
In the explicit form, the relation
$
\begin{bmatrix}0&\eurL\sb{-}\\-\eurL\sb{+}&0\end{bmatrix}
\begin{bmatrix}R\\S\end{bmatrix}
=
\lambda
\begin{bmatrix}R\\S\end{bmatrix}
$
can be written as follows:
\begin{equation}
\begin{bmatrix}
-\lambda&0&g-\omega&\p\sb x
\\
0&-\lambda&-\p\sb x&-g-\omega
\\
-m+\omega+f+2f'v^2&-\p\sb x-2f'vu&-\lambda&0
\\
\p\sb x-2f'vu&m+\omega-f+2f'u^2&0&-\lambda
\end{bmatrix}
\begin{bmatrix}R\sb 1\\R\sb 2\\S\sb 1\\S\sb 2\end{bmatrix}
=0.
\end{equation}
Dividing the first and the third rows by $\epsilon^2$,
the second and the fourth rows by $\epsilon$,
and substituting
$y=\epsilon x$,
$R\sb 2=\epsilon\hat R\sb 2$,
$S\sb 2=\epsilon\hat S\sb 2$,
we get
\begin{equation}\label{lhpm3}
\begin{bmatrix}
-\frac{\lambda}{\epsilon^2}&0&\frac{m-\omega-f}{\epsilon^2}&\frac{1}{\epsilon}\p\sb y
\\
0&-\frac{\lambda}{\epsilon}&-\p\sb y&\frac{1}{\epsilon}(-m-\omega+f)
\\
\frac{-m+\omega+f+2f'v^2}{\epsilon^2}&-\frac{1}{\epsilon}\p\sb y-\frac{2f'vu}{\epsilon^2}&-\frac{\lambda}{\epsilon^2}&0
\\
\p\sb y-\frac{1}{\epsilon}2f'vu&\frac{1}{\epsilon}(m+\omega-f+2f'u^2)&0&-\frac{\lambda}{\epsilon}
\end{bmatrix}
\begin{bmatrix}R\sb 1\\\epsilon\hat R\sb 2\\S\sb 1\\\epsilon\hat S\sb 2\end{bmatrix}
=0,
\end{equation}
which simplifies to
\begin{equation}\label{lhpm4}
\begin{bmatrix}
-\frac{\lambda}{\epsilon^2}&0&\frac{m-\omega-f}{\epsilon^2}&\p\sb y
\\
0&-\lambda&-\p\sb y&-m-\omega+f
\\
\frac{-m+\omega+f+2f'v^2}{\epsilon^2}&-\p\sb y-\frac{2f'vu}{\epsilon}&-\frac{\lambda}{\epsilon^2}&0
\\
\p\sb y-\frac{1}{\epsilon}2f'vu&m+\omega-f+2f'u^2&0&-\lambda
\end{bmatrix}
\begin{bmatrix}R\sb 1\\\hat R\sb 2\\S\sb 1\\\hat S\sb 2\end{bmatrix}
=0.
\end{equation}
Let $\Lambda=\lim\limits\sb{\epsilon\to 0}\frac{\lambda}{\epsilon^2}$.
We introduce the matrices
\begin{equation}\label{def-akk}
\eub{A}\sb\Lambda=\begin{bmatrix}
-\Lambda
&0&\frac{1}{2}-U^k&\p\sb y
\\
0&0&-\p\sb y&-2
\\
-\frac 1 2 +(2k+1)U^k
&-\p\sb y&-\Lambda&0
\\
\p\sb y&2&0&0
\end{bmatrix},
\qquad
\eub{K}\sb 1=\mathop{\rm diag}[1,0,1,0],
\qquad
\eub{K}\sb 2=\mathop{\rm diag}[0,1,0,1].
\end{equation}
Above,
\begin{equation}\label{def-vk-0}
U(y)
=
\Big(\frac{k+1}{2\cosh^2 ky}\Big)^{1/k};
\end{equation}
cf. \eqref{def-vk}.
Noting that one has
$\mathscr{X}(x)
=\epsilon^{\frac{2}{k}}U(\epsilon x)+O(\epsilon^{\frac{4}{k}})$
by Lemma~\ref{lemma-xx-ff},
we can write
\eqref{lhpm4} in the form
\begin{equation}\label{lhpm5}
\eub{A}\sb\Lambda
\eta
=
\Big(\frac{\lambda}{\epsilon^2}-\Lambda\Big)\eub{K}\sb 1\eta
+\lambda \eub{K}\sb 2\eta
+W\eta,
\end{equation}
where
$
\eta=\begin{bmatrix}R\sb 1\\\hat R\sb 2\\S\sb 1\\\hat S\sb 2\end{bmatrix}
\in\C^4
$
and
\begin{equation}\label{def-ww}
W(y,\epsilon)=
\eub{A}\sb\Lambda
-
\begin{bmatrix}
-\frac{\lambda}{\epsilon^2}&0&\frac{m-\omega-f}{\epsilon^2}&\p\sb y
\\
0&-\lambda&-\p\sb y&-m-\omega+f
\\
\frac{-m+\omega+f+2f'v^2}{\epsilon^2}&-\p\sb y-\frac{2f'vu}{\epsilon}&-\frac{\lambda}{\epsilon^2}&0
\\
\p\sb y-\frac{2f'vu}{\epsilon}&m+\omega-f+2f'u^2&0&-\lambda
\end{bmatrix}
-\Big(\frac{\lambda}{\epsilon^2}-\Lambda\Big)\eub{K}\sb 1
-\lambda \eub{K}\sb 2
\end{equation}
is the zero order differential operator with
$L\sp\infty$ coefficients.

\begin{lemma}\label{lemma-bounds-w}
$
\norm{W(\cdot,\epsilon)}\sb{L\sp\infty(\R,\C^4\to\C^4)}\le O(\epsilon^{2/k}).
$
\end{lemma}

\begin{proof}
By Lemma~\ref{lemma-xx-ff},
one has
\[
v(x,\omega)
=
\epsilon^{\frac 1 k}(U(\epsilon x))^{\frac 1 2}
+
O(\epsilon^{\frac 3 k}),
\qquad
\norm{u(\cdot,\omega)}\sb{L\sp\infty}=O(\epsilon^{1+\frac 1 k}).
\]
Then
\[
f(v^2-u^2)
=v^{2k}+O(\epsilon^{2+\frac 2 k})
=\epsilon^{2}U(\epsilon x)^k+O(\epsilon^{2+\frac 2 k}),
\]
\[
f'(v^2-u^2)=kv^{2k-2}+O(\epsilon^{2})=O(\epsilon^{2-\frac 2 k}),
\]
\[
f'(v^2-u^2)v^2
=
\epsilon^2 k U(\epsilon x)^k
+O(\epsilon^{2+\frac 2 k}),
\]
\[
f'(v^2-u^2)v u=O(\epsilon^{3}).
\]
Now the proof follows
from the definition of $\eub{A}\sb\Lambda$
and $\eub{K}\sb 1$, $\eub{K}\sb 2$
in \eqref{def-akk}.
\end{proof}

\begin{lemma}\label{lemma-ker-a}
$\dim\ker \eub{A}\sb\Lambda>0$ if and only if $\Lambda$
is an eigenvalue of the operator
\[
\eub{j}\eub{H}=\begin{bmatrix}
0&-\frac 1 2\p\sb y^2+\frac 1 2-U^k
\\
\frac 1 2\p\sb y^2-\frac 1 2+(2k+1)U^k
&0
\end{bmatrix}.
\]
When $k\ge 3$,
there is $\Lambda>0$
such that
$\pm\Lambda\in\sigma\sb{d}(\eub{j}\eub{H})$.
\end{lemma}

\begin{proof}
The first statement of the Lemma
follows from the structure of the second and the fourth
rows of $\eub{A}\sb\Lambda$.
The second statement follows from
Proposition~\ref{prop-a-hat}.
\end{proof}

By Lemma~\ref{lemma-ker-a},
if $k\ge 3$,
then there is $\Lambda>0$
such that
$\pm\Lambda\in\sigma\sb{d}(\eub{j}\eub{H})$.
Now we will use the
Rayleigh-Schr\"odinger perturbation theory
to show that
there are $\lambda\in\sigma\sb{d}(\eubJ\eubL)$
with $\lambda\approx\pm\epsilon^2\Lambda$.
Let $\Phi\sb\Lambda\in\ker \eub{A}\sb\Lambda$
with $\norm{\Phi\sb\Lambda}\sb{L\sp 2}=1$,
and
let $\eub{P}\sb\Lambda:\;L\sp 2(\R,\C^4)\to H\sp\infty(\R,\C^4)$
be the spectral projector
onto the corresponding eigenspace.
Coupling \eqref{lhpm5}
with $\eubJ\Phi\sb\Lambda$
and projecting \eqref{lhpm5}
onto $\range(1-\eub{P}\sb\Lambda)$,
one has:
\begin{equation}\label{eq1}
0=\langle \eubJ\Phi\sb\Lambda,\eub{A}\sb\Lambda\eta\rangle
=
\Big(\frac{\lambda}{\epsilon^2}-\Lambda\Big)
\langle \eubJ\Phi\sb\Lambda,\eub{K}\sb 1\eta\rangle
+
\lambda\langle \eubJ\Phi\sb\Lambda,\eub{K}\sb 2\eta\rangle
+\langle \eubJ\Phi\sb\Lambda,W\eta\rangle,
\end{equation}
\begin{equation}\label{eq2}
(1-\eub{P}\sb\Lambda)\eta=
\eub{A}\sb\Lambda^{-1}(1-\eub{P}\sb\Lambda)
\Big(
\Big(\frac{\lambda}{\epsilon^2}-\Lambda\Big)\eub{K}\sb 1
+\lambda \eub{K}\sb 2
+W
\Big)\eta.
\end{equation}
We will find $\eta$ in the form
$\eta=\Phi\sb\Lambda+\zeta$,
with $\zeta\in\range(1-\eub{P}\sb\Lambda)$.
Denote
$\mu=\frac{\lambda}{\epsilon^2}-\Lambda$.
Then the relations \eqref{eq1} and \eqref{eq2}
can be written as
\begin{equation}\label{eq1a}
0=
\mu
\langle \eubJ\Phi\sb\Lambda,\eub{K}\sb 1\Phi\sb\Lambda\rangle
+
\mu
\langle \eubJ\Phi\sb\Lambda,\eub{K}\sb 1\zeta\rangle
+
\epsilon^2
(\Lambda+\mu)
\langle \eubJ\Phi\sb\Lambda,\eub{K}\sb 2(\Phi\sb\Lambda+\zeta)\rangle
+
\langle \eubJ\Phi\sb\Lambda,W(\Phi\sb\Lambda+\zeta)\rangle,
\end{equation}
\begin{equation}\label{eq2a}
\zeta=
\eub{A}\sb\Lambda^{-1}(1-\eub{P}\sb\Lambda)
\Big(
\mu \eub{K}\sb 1
+\epsilon^2(\Lambda+\mu)\eub{K}\sb 2
+W
\Big)(\Phi\sb\Lambda+\zeta).
\end{equation}
Let us note that
$\langle \eubJ\Phi\sb\Lambda,\eub{K}\sb 1\Phi\sb\Lambda\rangle\ne 0$.
The equations
\eqref{eq1a}, \eqref{eq2a}
could be written as
$
\mu=M(\mu,\zeta),
$
$
\zeta=Z(\mu,\zeta),
$
with functions
$
M:\R\times L^2(\R,\C^4)\to\R,
$
$
Z:\R\times L^2(\R,\C^4)\to L^2(\R,\C^4)
$
given by
\begin{equation}\label{def-map-m}
M(\mu,\zeta)=
-\frac{1}
{\langle \eubJ\Phi\sb\Lambda,\eub{K}\sb 1\Phi\sb\Lambda\rangle}
\Big[
\mu
\langle \eubJ\Phi\sb\Lambda,\eub{K}\sb 1\zeta\rangle
+
\epsilon^2
(\Lambda+\mu)
\langle\eubJ\Phi\sb\Lambda,\eub{K}\sb 2(\Phi\sb\Lambda+\zeta)\rangle
+
\langle\eubJ\Phi\sb\Lambda,W(\Phi\sb\Lambda+\zeta)\rangle
\Big],
\end{equation}
\begin{equation}\label{def-map-z}
Z(\mu,\zeta)=
\eub{A}\sb\Lambda^{-1}(1-\eub{P}\sb\Lambda)
\Big(
\mu \eub{K}\sb 1
+\epsilon^2(\Lambda+\mu)\eub{K}\sb 2
+W
\Big)(\Phi\sb\Lambda+\zeta).
\end{equation}

Pick $\Gamma\ge 1$
such that
\begin{equation}\label{def-gamma}
\Gamma
\ge 2\norm{\eub{A}\sb\Lambda^{-1}(1-\eub{P}\sb\Lambda)\eub{K}\sb 1\Phi\sb\Lambda}\sb{L\sp 2}.
\end{equation}

\begin{lemma}\label{lemma-mz-contraction}
Consider $\R\times L\sp 2(\R,\C^4)$
endowed with the metric
\[
\norm{(\mu,\zeta)}\sb{\Gamma}
=\Gamma\abs{\mu}+\norm{\zeta}\sb{L\sp 2(\R,\C^4)}.
\]
There is $\omega\sb 1\in(\omega\sb 0,m)$
such that
for $\omega\in(\omega\sb 1,m)$
the map
\begin{equation}\label{def-m-z}
M\times Z:\;\;\R\times L\sp 2(\R,\C^4)\to\R\times L\sp 2(\R,\C^4),
\qquad
(\mu,\zeta)\mapsto\big(M(\mu,\zeta),Z(\mu,\zeta)\big),
\end{equation}
restricted onto the set
\[
\mathcal{B}\sb{\epsilon^{1/k}}
=\{(\mu,\zeta)\in\R\times L\sp 2(\R,\C^4)
\sothat\norm{(\mu,\zeta)}\sb{\Gamma}\le\epsilon^{1/k}\}
\subset\R\times L\sp 2(\R,\C^4)
\]
is an endomorphism and a contraction
with respect to $\norm{\cdot}\sb{\Gamma}$.
\end{lemma}

\begin{proof}
The proof follows from the definitions
\eqref{def-map-m}, \eqref{def-map-z}.
We needed to introduce the factor $\Gamma$
into the definition of the metric
to make sure that the contribution
of the term
$\eub{A}\sb\Lambda^{-1}(1-\eub{P}\sb\Lambda)\mu \eub{K}\sb 1\Phi\sb\Lambda$
from \eqref{def-map-z}
into $\norm{(M(\mu,\zeta),Z(\mu,\zeta))}\sb{\Gamma}$
is bounded by
$\frac \Gamma 2\abs{\mu}\le\frac 1 2\norm{(\mu,\zeta)}\sb{\Gamma}$.
The contribution
of all other terms
is bounded by $\frac 1 4\norm{(\mu,\zeta)}\sb{\Gamma}
+O(\epsilon^{2/k})$,
placing $(M(\mu,\zeta),Z(\mu,\zeta))$
inside $\mathcal{B}\sb{\epsilon^{1/k}}$
whenever $(\mu,\zeta)\in\mathcal{B}\sb{\epsilon^{1/k}}$,
as long as $\omega\sb 1$ is sufficiently close to $m$,
so that $\epsilon=\sqrt{m^2-\omega^2}$ is sufficiently small.
The contribution $O(\epsilon^{2/k})$
comes from the terms with $W\Phi\sb\Lambda$,
due to the bound
$\norm{W(\cdot,\epsilon)}\sb{L\sp\infty(\R,\C^4\to\C^4)}=O(\epsilon^{2/k})$
established in Lemma~\ref{lemma-bounds-w}.

The contribution
of the term
$\eub{A}\sb\Lambda^{-1}(1-\eub{P}\sb\Lambda)\mu \eub{K}\sb 1\Phi\sb\Lambda$
into the norm $\norm{Z(\mu,\zeta)-Z(\mu',\zeta')}\sb{L\sp 2}$
is bounded by $\frac 1 2 \Gamma\abs{\mu-\mu'}
\le\frac 1 2\norm{(\mu-\mu',\zeta-\zeta')}\sb{\Gamma}$.
The contribution of all other terms
from
\eqref{def-map-m}, \eqref{def-map-z}
into
$\norm{(M(\mu,\zeta)-M(\mu',\zeta'),\, Z(\mu,\zeta)-Z(\mu',\zeta'))}\sb{\Gamma}$
could be made smaller than
$\frac 1 4\norm{(\mu-\mu',\zeta-\zeta')}\sb{\Gamma}$
by choosing
$\omega\sb 1$ sufficiently close to $m$,
so that $\epsilon=\sqrt{m^2-\omega^2}$ is sufficiently small.
It follows that
$(M\times Z)\at{\mathcal{B}\sb{\epsilon^{1/k}}}$
is a contraction
in the metric $\norm{\cdot}\sb{\Gamma}$.
\end{proof}

According to Lemma~\ref{lemma-mz-contraction},
by the contraction mapping theorem,
for $\omega\in(\omega\sb 1,m)$,
the map
\eqref{def-m-z}
has a fixed point
\[
(\mu\sb 0(\omega),\zeta\sb 0(\omega))
\in
\mathcal{B}\sb{\epsilon^{1/k}}
\subset\R\times L^2(\R,\C^4).
\]
Thus, we have
\[
\pm\epsilon^2(\Lambda+\mu\sb 0(\omega))
\in\sigma\sb p(\eubJ\eubL(\omega)),
\qquad
\omega\in(\omega\sb 1,m),
\]
with $\Gamma\abs{\mu\sb 0(\omega)}\le\epsilon^{1/k}$,
finishing the proof of the proposition.
\end{proof}

By Remark~\ref{remark-enough},
Proposition~\ref{prop-k3}
finishes the proof of Theorem~\ref{main-theorem-dirac-vbk}.


\newcommand{\etalchar}[1]{$^{#1}$}
\def\cprime{$'$} \def\cprime{$'$} \def\cprime{$'$} \def\cprime{$'$}
  \def\cprime{$'$} \def\cprime{$'$} \def\cprime{$'$} \def\cprime{$'$}
  \def\cprime{$'$} \def\cydot{\leavevmode\raise.4ex\hbox{.}} \def\cprime{$'$}
  \def\cydot{\leavevmode\raise.4ex\hbox{.}} \def\cprime{$'$} \def\cprime{$'$}
  \def\cprime{$'$} \def\cprime{$'$}

\end{document}